\documentclass[11pt]{amsart}

\usepackage{geometry}
\usepackage{todonotes}

\usepackage{amsfonts, amssymb, amscd}
\numberwithin{equation}{section}

\usepackage[symbol]{footmisc}

\usepackage{bm}
\usepackage{verbatim}
\usepackage{mathrsfs}
\usepackage{graphicx}
\usepackage{tikz-cd}
\usepackage{subcaption}
\usepackage{listings}
\usepackage{subfiles}
\usepackage[toc,page]{appendix}
\usepackage{mathtools}
\usepackage{comment}
\usepackage{enumerate}
\usepackage{enumitem}
\usepackage[all]{xy}

\usepackage{graphicx}
\graphicspath{{images/}}

\usepackage{appendix}
\usepackage{hyperref}
\hypersetup{
    colorlinks=true,
    citecolor=red,
    linkcolor=blue,
    filecolor=magenta,      
    urlcolor=red,
}
\newcommand{\bQ}{\mathbb{Q}}
\newcommand{\bP}{\mathbb{P}}

\newcommand{\vol}{\operatorname{vol}}

\newtheorem{definition}{Definition}[section]
\newtheorem{lemma}[definition]{Lemma}
\newtheorem{theorem}[definition]{Theorem}
\newtheorem{proposition}[definition]{Proposition}
\newtheorem{corollary}[definition]{Corollary}
\newtheorem{remark}[definition]{Remark}

\theoremstyle{definition}
\newtheorem{remark}{Remark}[section]

\begin{document}

\title{Explicit effective birationality for singular surfaces}

\author{Pinxian Bie}

\address{\rm School of Mathematical Sciences, Fudan University, Shanghai 200433, China}
\email{23110180002@m.fudan.edu.cn}

\date{\today}

\begin{abstract}
We give some explicit upper bounds on the effective birationality of the canonical or anti-canonical system for a singular surface. In particular, we show that for any surface $X$ with $\epsilon$-lc singularity and the canonical divisor $K_X$ or the anti-canonical divisor $-K_X$ is big and nef, then there exists explicit natural numbers $m$ and $l$ depending only on $\epsilon$ such that $|mK_X|$ or $|-lK_X|$ defines a birational map. Although these explicit values are expected to be far from optimal, they are the first explicit upper bounds of this type for surfaces. 
\end{abstract}

\maketitle
\tableofcontents

\section{Introduction}

We work over the field of complex number $\mathbb C$. 

In his celebrated work \cite{Bir19}, Birkar established the effective birationality for $\epsilon$-lc Fano varieties. This result was later extended to arbitrary $\epsilon$-lc polarized varieties in \cite{Bir23}. Both papers rely on the notion of potential birationality of divisors and numerous properties of bounded families, rendering the explicit determination of a bound for effective birationality highly challenging. It is therefore very interesting to explore whether explicit bounds for the integer $m$ can be given, at least in low dimensions. Some progress has been made in this direction for surfaces and threefolds with at most canonical singularities \cite{CJ16, CJ20, CC15}. However, few results \cite{Zhu23} address varieties with singularities worse than canonical. In this paper, we derive explicit upper bounds for $m$ in the case of $\epsilon$-log canonical surfaces, considering both the weak Fano case and the case of general type.

\begin{theorem}\label{thm: weak fano surface explicit birationality}
Let $X$ be a $\epsilon$-lc projective weak Fano surface, then let $m_0=96\lfloor2\left(\frac{2}{\epsilon}\right)^{\frac{128}{\epsilon^5}}\rfloor !$, $|-km_0K_X|$ always defines a birational morphism for all $k\geq (2+\lfloor\frac{4}{\epsilon}\rfloor)$. In particular, let $m_1=96(2+\lfloor\frac{4}{\epsilon}\rfloor)\lfloor2\left(\frac{2}{\epsilon}\right)^{\frac{128}{\epsilon^5}}\rfloor !$, then $|-m_1K_X|$ defines a birational morphism. Moreover, if $m_2=5m_1-1$, then for every $m'\geq m_2$, $|-m'K_X|$ defines a birational map.
\end{theorem}

We have the following corollary, which is quickly implied by \ref{thm: weak fano surface explicit birationality} and the fact that exceptional Fano surfaces are $\frac{1}{42}$-lc.

\begin{corollary}\label{cor: 1/42 exceptional fano surface}
For any exceptional Fano surface $X$, take $m_0=96\lfloor2\cdot (84)^{128\cdot 42^5}\rfloor !$, then the linear system $|-km_0K_X|$ always defines a birational morphism for all $k\geq 170$. Moreover, for any $m'\geq 850m_0-1$, $|-m'K_X|$ defines a birational map.
\end{corollary}

The situation of the general type case is very different, because we can't hope to give an explicit $m$ such that the linear system $|mK_X|$ defines a birational morphism for all such $X$, as the Cartier index of $K_X$ is not bounded. This also gives a hint that the methods we use to prove Theorem \ref {thm: weak fano surface explicit birationality} and the following Theorem \ref {thm: general type surface explicit birationality} are quite different.

\begin{theorem}\label{thm: general type surface explicit birationality}
Let $X$ be a projective $\epsilon$-lc surface of general type, such that $K_X$ is big and nef. Then let $m_0=\max \{2v,\lfloor\sqrt{42\cdot 84^{128\cdot42^5+168}\cdot v} \rfloor\}$ where $$v=64(\max\{\lceil\frac{192}{\epsilon}\rceil+1,(45(\lceil\frac{2}{\epsilon}\rceil!)^2+1)\})^2$$ such that there exists some $m^*\leq m_0$, $|m^*K_X|$ defines a birational map. In particular, $|(m_0!)K_X|$ defines a birational map for all such $X$. Moreover, for every $m'\geq 5(m_0!)+1$, $|m'K_X|$ defines a birational map for all $X$.
\end{theorem}

We also derive a result about explicit pseudo-effective threshold of nef and big divisors on $\epsilon$-lc surfaces. It's useful in some more general settings.

\begin{theorem}\label{pseudo effective threshold}
    Let $\epsilon$ and $\delta$ be two positive real numbers. Then there is a natural number $$l=\lceil1+\frac{\lfloor2\left(\frac{2}{\epsilon}\right)^{\frac{128}{\epsilon^5}}\rfloor !)\cdot \max\{64,\frac{8}{\epsilon}+4\}}{\min\{1,\delta\}}\rceil$$ satisfying the following. Assume
    \begin{enumerate}
        \item $X$ is an $\epsilon$ -lc projective surface,
        \item $X\rightarrow Z$ is a contraction,
        \item $N$ is an $\mathbb{R}$-divisor on $X$ which is nef and big over $Z$, and
        \item $N=E+R$ where $E$ is integral and pseduo-effective and $R\geq 0$ with coefficients in ${0}\cup [\delta,+\infty)$. 
    \end{enumerate}

    Then $K_X+lN$ is big over $Z$.
\end{theorem}

Although the bounds in Theorem \ref{thm: weak fano surface explicit birationality} and Theorem \ref{thm: general type surface explicit birationality} are expected to be far from being optimal, these are the first precise upper bounds for such polarized surfaces. Similar topics include the boundedness of the the anti-canonical volume of Fano varieties (cf.\cite{Pro05, Pro07, CC08, Pro10, Che11, Pro13, CJ16, CJ20, JZ21, Jia21b, JZ22, Bir22}), estimation of $(\epsilon,n)$-complement \cite{CH21}, the explicit M\textsuperscript{c}Kernan-Shokurov conjecture \cite{HJL22}, precise bounds of mlds \cite{Jia21a,LX21,LL22}, explicit boundedness of complements for Fano varieties \cite{Liu23, BL23} etc.

\medskip

\noindent\textbf{Acknowledgement}. The author expresses his gratitude to his advisor Professor Meng Chen for his great support and encouragement. The author would also like to thank Sicheng Ding, Zhengjie Yu, Peien Du and Hanye Gu for useful discussions. 

\medskip

\section{Preliminaries}
\subsection{Basic definitions}

We adopt the standard notations in \cite{KM98,BCHM10} and will freely use them. We recall some of them in this subsection.

\begin{definition}[Volume]
Let \(X\) be a projective variety of dimension \(n\) and \(D\) be a \(\mathbb{Q}\)-divisor on \(X\). Define the volume of \(D\) as
\[
\operatorname{vol}(D) = \limsup_{m \to \infty} \frac{h^0(X, \mathcal{O}_X(\lfloor mD \rfloor))}{m^n / n!}.
\]
We say $D$ is \textit{big} if \(\operatorname{vol}(D) >0\). 
If \(D\) is nef, then \(\operatorname{vol}(D) = D^n\).
\end{definition}

\begin{definition}[Singularities]
 A \textit{sub-pair} \((X,B)\) consists of a normal quasi-projective variety \(X\) and an \(\mathbb{R}\)-divisor \(B\) such that \(K_{X}+B\) is \(\mathbb{R}\)-Cartier. If the coefficients of \(B\) are at most 1 we say \(B\) is a \textit{sub-boundary}, and if in addition \(B\geq 0\), we say \(B\) is a \textit{boundary}. A sub-pair \((X,B)\) is called a \textit{pair} if \(B\geq 0\). In this paper, we are mainly focused on the case where $B=0$, in this case we assume $K_X$ to be $\bQ$-Cartier.

Let \(\phi\colon W\to X\) be a log resolution of a sub-pair \((X,B)\). Let \(K_{W}+B_{W}\) be the pullback of \(K_{X}+B\). The \textit{log discrepancy} of a prime divisor \(D\) on \(W\) with respect to \((X,B)\) is \(1-\mu_{D}B_{W}\) and it is denoted by \(a(D,X,B)\). We say \((X,B)\) is \textit{sub-lc} (resp. \textit{sub-klt}) (resp. \textit{sub-\(\epsilon\)-lc}) if \(a(D,X,B)\) is \(\geq 0\) (resp. \(>0\)) (resp. \(\geq\epsilon\)) for every \(D\). When \((X,B)\) is a pair we remove the sub and say the pair is lc, etc. Note that if \((X,B)\) is a lc pair, then the coefficients of \(B\) necessarily belong to \([0,1]\). Also if \((X,B)\) is \(\epsilon\)-lc, then automatically \(\epsilon\leq 1\) because \(a(D,X,B)=1\) for most \(D\).

Let \((X,B)\) be a sub-pair. A \textit{non-klt place} of \((X,B)\) is a prime divisor \(D\) on birational models of \(X\) such that \(a(D,X,B)\leq 0\). A \textit{non-klt centre} is the image on \(X\) of a non-klt place. When \((X,B)\) is lc, a non-klt centre is also called an \textit{lc centre}.
\end{definition}

\subsection{Potentially birational divisors}

\begin{definition}
Let \(X\) be a normal projective variety and let \(D\) be a big \(\mathbb{Q}\)-Cartier \(\mathbb{Q}\)-divisor on \(X\). We say that \(D\) is \textit{potentially birational} if for any pair \(x\) and \(y\) of general closed points of \(X\), possibly switching \(x\) and \(y\), we can find \(0\leq\Delta\sim_{\mathbb{Q}}(1-\epsilon)D\) for some \(0<\epsilon<1\) such that \((X,\Delta)\) is not klt at \(y\) but \((X,\Delta)\) is lc at \(x\) and \(\{x\}\) is an isolated non-klt centre.
\end{definition}

\begin{lemma}[{\cite[Lemma 2.3.4]{HMX13}}]\label{lem: potentially birationality}
Let $X$ be a normal projective variety and let $D$ be a big $\mathbb{Q}$-Cartier $\mathbb{Q}$-divisor.
\begin{enumerate}
    \item If \(D\) is potentially birational, then \(|K_{X}+\lceil D \rceil|\) defines a birational map.
    \item If $X$ has dimension $d$ and $\phi_D$ defines a birational map, then $(2d+1)\lfloor D\rfloor$ is potentially birational.

\end{enumerate}
\end{lemma}

\subsection{Minimal models, Mori fibre spaces, and MMP}

Let \(X \to Z\) be a projective morphism of normal varieties and \(D\) be an \(\mathbb{R}\)-Cartier \(\mathbb{R}\)-divisor on \(X\). Let \(Y\) be a normal variety projective over \(Z\) and \(\phi\colon X \dashrightarrow Y/Z\) be a birational map whose inverse does not contract any divisor. Assume \(D_{Y} := \phi_{*}D\) is also \(\mathbb{R}\)-Cartier and that there is a common resolution \(g\colon W \to X\) and \(h\colon W \to Y\) such that \(E := g^{*}D - h^{*}D_{Y}\) is effective and exceptional over \(Y\), and \(\operatorname{Supp} g_{*}E\) contains all the exceptional divisors of \(\phi\).

Under the above assumptions we call \(Y\) a \textit{minimal model} of \(D\) over \(Z\) if \(D_{Y}\) is nef over \(Z\). On the other hand, we call \(Y\) a \textit{Mori fibre space} of \(D\) over \(Z\) if there is an extremal contraction \(Y \to T/Z\) with \(-D_{Y}\) ample over \(T\) and \(\dim Y > \dim T\).

We will use standard results of the minimal model program (c.f.\cite{BCHM10}).

\subsection{Cartier index of singular Fano surfaces}

BAB states that the set of $\epsilon$-lc weak Fano varieties of dimension $d$ form a bounded family, thus the Cartier index $I$ of the canonical divisor $K_X$ for any such $X$ is bounded. It's also very interesting to try to determine an explicit upper bound of the Cartier index $I$ for some fixed $\epsilon$ and $d$. At least in dimension $2$, this is done by the following lemma.

\begin{lemma}[cf. {\cite[Proof of Lemma 3.7]{AM04},~\cite[After Theorem A]{Lai16}, \cite[Lemma 2.2]{Bir22}}]\label{lem: index epsilonlc fano surface}
Let $\epsilon$ be a positive real number and $X$ an $\epsilon$-lc weak Fano surface. Then $IK_X$ is Cartier for some positive integer
$I\leq 2\left(\frac{2}{\epsilon}\right)^{\frac{128}{\epsilon^5}}.$ In particular, let $I_0:= \lfloor2\left(\frac{2}{\epsilon}\right)^{\frac{128}{\epsilon^5}}\rfloor !$, then $I_0K_X$ is always Cartier for any such $X$.
\end{lemma}

\subsection{Explicit bounds for the volume of  surfaces}

As the set of volumes of any $d$ dimensional varieties with big canonical divisors satisfies the DCC, it is bounded away from 0. Thus it's natural to ask if one can find a positive lower bound of this set. For the surface case, an explicit lower bound was found in \cite{AM04}.

\begin{lemma}[{\cite[Section 10]{AL19}}]\label{lem: lower bounds of volumes}
For any log canonical surface $X$ such that $K_X$ is big, $\operatorname{Vol}(X)\geq \frac{1}{42\cdot 84^{128\cdot42^5+168}}$.
\end{lemma}

The above lower bound is obviously far from being optimal. Some minor improvements can be made if one estimate the proof of \cite{AM04} in detail, but it won't cardinally change the result of the estimation without introducing some other more effective methods.

We also mention an effective upper bound of the anti-canonical volume of $\epsilon$-lc weak Fano surfaces due to \cite{Lai16}.

\begin{lemma}[c.f. {\cite[Theorem A]{Lai16}}]
  Let $X$ be an $\epsilon$-lc weak Fano surface. The volume $\operatorname{Vol}(-K_X)=(K_X^2)$ satisfies:

  $$(K_X^2)\leq \max\{64,\frac{8}{\epsilon}+4\}.$$ Moreover, this upper bound is in a sharp form.
\end{lemma}

\subsection{Effective birationality with additional conditions}
The following theorem proved by Minzhe Zhu in \cite{Zhu23} shows that under some additional conditions, the explicit effective bound of the integer $m$ is easier to find.

\begin{theorem}[{\cite[Corollary 1.7]{Zhu23}}]\label{thm: effective birationality with h^0 > 2}
    Let \( S \) be a projective \(\epsilon\)-lc surface. Let \( H \) be a nef and big Weil divisor on \( S \) with \( h^0(H) \geq 2 \) and \( L \) be an effective Weil divisor satisfying one of the following conditions:
\begin{enumerate}
    \item \( |L - K_S| \neq \emptyset \);
    \item \( L - K_S \) is nef.
\end{enumerate}
Then \( \vol(H) \geq \frac{\epsilon}{2} \) and \( |L + mH| \) is birational for \( m \geq 2 + \left\lfloor \frac{4}{\epsilon} \right\rfloor \).
\end{theorem}

\begin{remark}
The above theorem can be easily generalized to the following form: $$\operatorname{Vol}(H)\geq \frac{\epsilon}{2}(h^0(H)-1).$$ In particular, if $X$ is of general type, by taking $H=K_X$, we obtain the following Noether-type inequality:
\begin{equation}
    \operatorname{Vol}(X)=K_X^2\geq \frac{\epsilon}{2}(p_g(X)-1).
\end{equation}

Moreover, this inequality is sharp for $\epsilon=\frac{2}{3}$ and $p_g=2$, as there exists a weighted surface $S=S_{18}\subset  \mathbb{P}(1,1,6,9)$ with only cyclic quotient singularity of type $\frac{1}{3}(1,1)$ such that $S$ is $\frac{2}{3}$-lc and $p_g(S)=2$, $\operatorname{Vol}(S)=\frac{18}{6\cdot 9}=\frac{1}{3}$.

However, the above inequality is certainly not optimal for $\epsilon$ is small or $p_g$ is large. Indeed, by \cite{Liu25}, for any projective log canonical surface pair $(X,B)$, let $c:=\operatorname{min}(\mathcal{C}_B\cup\{1\})$ where $\mathcal{C}_B$ is the set of non-zero coefficients of $B$, they proved the following much stronger Noether type inequality which is optimal in their sense:

\begin{equation}
    \operatorname{Vol}(K_X+B)\geq (2c-c^2)p_g(X,B)-(2c+c^2).
\end{equation}
\end{remark}

\subsection{Adjunction and the change of coefficients} We will use the adjunction which was introduced in \cite{HMX14} and developed in detail in \cite{Bir19}. We refer the readers to these two standard papers for the detailed construction. All the constructions and results in this subsection work for any dimensions.

Assume the setting as follows:

\begin{itemize}
    \item \((X,B)\) is a projective klt pair,
    \item \(G\subset X\) is a subvariety with normalisation \(F\),
    \item \(X\) is \(\mathbb{Q}\)-factorial near the generic point of \(G\),
    \item \(\Delta\geq 0\) is an \(\mathbb{R}\)-Cartier divisor on \(X\), and
    \item \((X,B+\Delta)\) is lc near the generic point of \(G\), and there is a unique non-klt place of this pair with centre \(G\) (but the pair may have other non-klt places whose centres are not \(G\)).
\end{itemize}

Then we can define an \(\mathbb{R}\)-divisor \(\Theta_{F}\) on \(F\) with coefficients in \([0,1]\) giving an adjunction formula
\[
K_{F}+\Theta_{F}+P_{F}\sim_{\mathbb{R}}(K_{X}+B+\Delta)|_{F}
\]
where in general \(P_{F}\) is pseudo-effective and determined only up to \(\mathbb{R}\)-linear equivalence. Moreover, if the coefficients of \(B\) are contained in a fixed DCC set \(\Phi\), then the coefficients of \(\Theta_{F}\) are also contained in a fixed DCC set \(\Psi\) depending only on \(\dim X\) and \(\Phi\)~\cite[Theorem 4.2]{HMX14}.

We also need the following property of descent of divisor coefficients to non-klt centres, which is a key ingredient of our proof.

\begin{proposition}[{\cite[Proposition 3.7]{Bir22}}]\label{prop: explicit coefficient descent}
    Let \(\epsilon\) be a positive real number. Then there is a natural number \(q\) depending only on \(\epsilon\) satisfying the following. Let \((X,B),\Delta,G,F,\Theta_{F},P_{F}\) be as in the above assumption. Assume in addition that

\begin{itemize}
    \item \(P_{F}\) is big and for any choice of \(P_{F}\geq 0\) in its \(\mathbb{R}\)-linear equivalence class the pair \((F,\Theta_{F}+P_{F})\) is \(\epsilon\)-lc,
    \item \(\Phi\subset\mathbb{R}\) is a subset closed under addition,
    \item \(E\) is an \(\mathbb{R}\)-Cartier \(\mathbb{R}\)-divisor on \(X\) with coefficients in \(\Phi\), and
    \item \(G\not\subset\operatorname{Supp}E\).
\end{itemize}

Then \(qE|_{F}\) has coefficients in \(\Phi\). Moreover, $q$ can be explicitly taken to be $q=(\lceil\frac{1}{\epsilon}\rceil!)^2$.
\end{proposition}

\begin{proof}
    The first part is just proposition 3.7 of \cite{Bir23}, while $q=(\lceil\frac{1}{\epsilon}\rceil!)^2$ also follows from a detailed tracking of the proof of that proposition. Indeed, if we use the same notation as its proof, we have $q=pl$, where $p=\lfloor\frac{1}{\epsilon}\rfloor!$. On the other hand, $l$ can be explicitly determined as $l=\lceil\frac{1}{\epsilon}\rceil!$ in the proof of lemma 3.5 of \cite{Bir23}, thus the claim follows.
\end{proof}

\section{The weak Fano case}
This case is easier, as we have an explicit upper bound for the Cartier index \ref{lem: index epsilonlc fano surface}, which enables us to use the powerful effective base-point-free theorem.

\subsection{Proof of Theorem \ref{thm: weak fano surface explicit birationality}}
\begin{proof}
    First of all, by Lemma \ref{lem: index epsilonlc fano surface}, we have $I_0K_X$ is always a Cartier divisor. By the effective base-point-freeness theorem (\cite[Theorem 1.1, Remark 1.2]{Fuj09} \cite[1.1 Theorem]{Kol93}), we have $|-96I_0K_X|$ is base-point-free. As $-K_X$ is big, this implies $$h^0(-96I_0K_X)\geq2.$$ 
    
    Thus by Theorem \ref{thm: effective birationality with h^0 > 2}, $|-km_0K_X|$ is a birational map for all $k\geq (2+\lfloor\frac{4}{\epsilon}\rfloor)$, where $m_0=96I_0$. Moreover it is actually a birational morphism as the corresponding linear system is base-point-free. Let $m_1$ and $m_2$ be taken as in the theorem, it suffice to check that for any $m'\geq m_2$, $|-m'K_X|$ defines a birational map. Note that $$-m'K_X =K_X +5(-m_1K_X) +(m'-m_2)(-K_X) $$ by Lemma \ref{lem: potentially birationality} (2) we have $5(-m_1K_X)$ is potentially birational and as $(m'-m_2)(-K_X)$ is big, $5(-m_1K_X) +(m'-m_2)(-K_X)$ is also potentially birational, which in turn implies $|-m'K_X|$ defines a birational map by Lemma \ref{lem: potentially birationality} (1).
\end{proof}
\begin{remark}
    The above argument actually works more generally with any nef and big Weil divisor $N$ on any $\epsilon$-lc Fano surfaces. In fact, for any Weil divisor $D$, $I_0D$ is always Cartier, thus we also have $(N^2)\geq \frac{1}{I_0}$ and $I_0N$ is a nef and big Cartier divisor, while the rest of the argument is similar. (Note that we have $$|m'N|=|K_X+5m_1N+(m'-5m_1)N-K_X|$$ is birational when $m'\geq 5m_1$.)
\end{remark}

\subsection{Proof of Corollary \ref{cor: 1/42 exceptional fano surface}}
\begin{proof}
    Recall that a Fano surface $X$ is said to be exceptional if for every $G\geq 0$ such that $K_X+G\sim_{\mathbb{R}}0$, $(X,G)$ is klt. It's known in \cite{Liu23} that exceptional Fano surfaces are always $\frac{1}{42}$. Thus we may take $\epsilon =\frac{1}{42}$ in Theorem \ref{thm: weak fano surface explicit birationality} and the assertion follows.
\end{proof}

\section{The general type case}

\subsection{A crucial lemma}

Before we start to treat the general type case, we state and prove a lemma which may be useful in a much more general setting. We use the basic idea of the proof of Proposition 4.5 of \cite{Bir23}, and the main take away here is to notice that everything can be treated explicitly in the case of surfaces.

\begin{lemma}[{\cite[Proposition 4.5]{Bir23}}]\label{lem: effective birationality for polarized surfaces} Let $\epsilon$, $\delta$ be two positive real numbers. Then there exists $v=64(\max\{\lceil\frac{192}{\epsilon}\rceil+1+ \lceil\frac{1}{\delta}\rceil,(1+\lceil\frac{1}{\delta}\rceil)((45(\lceil\frac{2}{\epsilon}\rceil!)^2+1))\})^2$ satisfying the following. Assume:
\begin{enumerate}
    \item $X$ is a projective $\epsilon$-lc surface,
    \item $N$ is a nef and big $\bQ$-divisor on $X$,
    \item $N-K_X$ and $N+K_X$ are big, and
    \item $N=E+R$ where $E$ is integral and pseudo-effective and $R\geq 0$ with coefficients in ${0}\cup [\delta,+\infty)$.
\end{enumerate}
 If $m$ is the smallest natural number such that $|mN|$ defines a birational map, then either $m\leq v$ or $\operatorname{Vol}(mN)\leq v$.
\end{lemma}

\begin{proof} We use the basic methods introduced in \cite{Bir19}.

 \emph{Step 1.} In this step we set up the notions and make some standard reductions.   
  By taking a $\bQ$-factorialisation we can assume $X$ is $\bQ$-factorial. Let $m$ be the smallest natural number such that $|mN|$ defines a birational map and let $n$ be the smallest natural number such that $\operatorname{Vol}(nN)>16$. Here $16$ is just $(2d)^d$ for dimension $d=2$.

  We claim that it suffice to show $\frac{m}{n}$ is bounded from above by an explicit number. Indeed, if we have $\frac{m}{n}\leq u$ for an explicit number $u$, then either $n=1$ which immediately gives $m\leq u$, or $n\geq 2$, we have $$\operatorname{Vol}(mN)=(\frac{m}{n-1})^2\operatorname{Vol}((n-1)N)\leq 64(\frac{m}{n})^2$$ which is explicitly bounded from above as well.\\

\emph{Step 2.} In this step we create a covering family of non-klt centres on $X$. As $\operatorname{Vol}(nN)>16$, by \cite{Bir23} 2.15(2), there exists a covering family of subvarieties of $X$ such that for any general colsed points $x,y\in X$, we can choose a member $G$ of the family and choose a $\bQ$-divisor 
$$0\leq \Delta\sim_{\bQ}nN$$ so that $(X,\Delta_0)$ is lc near $x$ with a unique non-klt place whose centre contains $x$, that centre is $G$, and $(X,\Delta_0)$ is not klt near $y$.

Since by assumption $N-K_X$ is big, we can fix some $$0\leq Q\sim_\bQ N-K_X$$ independent of the choice of $x,y$. As $x,y$ are general, we can assume $x,y$ are not contained in $\operatorname{Supp}Q$, thus we have $$0\leq \Delta := \Delta_0+Q \sim_\bQ (n+1)N-K_X $$ such that $(X,\Delta)$ is lc near $x$ with a unique non-klt place whose centre contains $x$, that centre is $G$, and $(X,\Delta)$ is not klt near $y$.

  As $x,y$ are general, we can assume $G$ is a general member in the above covering family of subvarieties. By definition of families of subvarieties, we have finitely many morphisms $V^j\rightarrow T^j$ of projective varieties and corresponding surjective morphisms $V^j\rightarrow X$ and that $G$ is just a general fiber of one of the $V^j\rightarrow T^j$. Moreover, we can always assume the points of $T^j$ corresponding to such $G$ are dense in $T^j$.

  Let $$d_G:=\operatorname{max}_j\{\operatorname{dim}V^j-\operatorname{dim}T^j\}.$$ It's clear $d_G\in \{0,1\}$. Assume first that $d_G=0$, which means $\operatorname{dim}G=0$ for all such $G$. This implies $(n+1)N-K_X$ is potentially birational. Take $r\in \mathbb{N}$ such that $r\delta \geq 1$ (for example we can take $r=\lceil\frac{1}{\delta}\rceil$). By assumption, $N= E+R$ with $E$ integral and pseudo-effective, and $R\geq \delta$. As the fraction part $R'$ of $(n+1+r)N-K_X$ is always supported in $R$ and $rR\geq 1$, we have $$\lfloor(n+1+r)N-K_X\rfloor = (n+1)N-K_X +rE +rR-R'$$ is potentially birational. Therefore, $$|\lfloor(n+1+r)N\rfloor| = |K_X+ \lfloor(n+1+r)N-K_X\rfloor|$$ defines a birational map by \ref{lem: potentially birationality}(1), which also implies $|(n+1+r)N|$ is birational, which by definition gives $m\leq n+1+r$ and thus $\frac{m}{n}\leq r+2$ is explicitly bounded in this case. Therefore we may assume $d_G =1$ for all such $G$ appearing as general fibers of $V^j\rightarrow T^j$ for some $j$.\\

  \emph{Step 3.} We now study this case in detail. Let $l$ be the smallest natural number such that $\operatorname{Vol}(lN|_G)>4$ for all the positive dimensional $G$, here $4$ is jus $d^d$ for dimension $d=2$. We claim that it suffice to bound $\frac{l}{n}$ from above instead. Indeed, if we know $\frac{l}{n}\leq a$ for some explicit natural number $a$, then for all such $G$ we have $$4<\operatorname{Vol}(lN|_G)\leq \operatorname{Vol}(anN|_G).$$ This helps us to slice down the dimension of $G$ by applying Lemma 2.18 of \cite{Bir23}. To be more precise, by replacing $n$ with $3an$ and modify the new $0\leq \Delta\sim_{\bQ}(3an+1)N-K_X$ and $G$, after possibly switching $x,y$, we can decrease the dimension of $d_G$, so we are reduced to the case of Step 2. In particular, by the same argument we have $m\leq 3an+1+r$ and thus $\frac{m}{n}\leq 3a+r+1$ is explicitly bounded. So we can from now on focus on bounding $\frac{l}{n}$ from above.\\

  \emph{Step 4.} From this step forward, our treatment is restricted only on dimension two. We still need to use adjunction on the normalization of $G$. Let $F$ be the normalization of $G$, since $G$ is general, $X$ is smooth near the generic point of $G$. By the adjunction mentioned in (2.6), we can write:
  $$K_F + \Delta_F := K_F+\Theta_F +P_F\sim_{\mathbb{R}}(K_X+\Delta)|_F \sim_{\mathbb{R}}(n+1)N|_F$$ where $\Theta_F\geq 0$ and $P_F$ is pseudo-effective. Since $x$ is a general point, we can pick some $0\leq N'\sim_\bQ N$ not containing $x$ and adding $N'$ to $\Delta$ won't change $\Theta_F$ but changes $P_F$ to be $P_F+N|_F$, so if we rewrite $0\leq \Delta \sim_\bQ (n+2)N-K_X$, we may assume $P_F$ is big.

On the other hand, by \cite{HMX14} Theorem 4.2, we can write $$K_F+\Lambda_F = K_X|_F.$$ It's clear that $(F,\Lambda_F)$ is sub-$\epsilon$-lc and $\Lambda_F\leq \Theta_F$ by \cite{Bir19} Lemma 3.12. It's important to note that in our situation, $F$ is a smooth curve, and $(F,\Lambda_F)$ is sub-$\epsilon$-lc simply means that the coefficients of each component of $\Lambda_F$ is equal or less than $1-\epsilon$.\\

\emph{Step 5.} In this step we show that we can reduce to the case where $(F,\Delta_F)$ is $\frac{\epsilon}{2}$ for any choice of $P_F\geq 0$ in its $\mathbb{R}$-linear equivalence class. Indeed, by Step 4, we can see:
$$\Delta_F-\Lambda_F= (K_X+\Delta)|_F-K_X|_F\sim_{\mathbb{R}}((n+2)N-K_X)|_F.$$ 

If $\operatorname{deg}(\Delta_F-\Lambda_F)\geq \frac{\epsilon}{2}$, we have the following:
$$ \frac{\epsilon}{2} \leq \operatorname{Vol}(((n+2)N-K_X)|_F) =  \operatorname{Vol}(((n+3)N-(K_X+N)|_G) \leq \operatorname{Vol}((n+3)N|_G)\leq 4\cdot \frac{n+3}{l-1}$$
where the first inequality holds because $(K_X+N)$ is big and the second inequality holds because $\operatorname{Vol}((l-1)N|_G)\leq 4$ by assumption. This in turn gives $$ \frac{l}{n} = \frac{l}{l-1}\cdot\frac{l-1}{n+3}\cdot\frac{n+3}{n}\leq \frac{64}{\epsilon}$$ which yields an explicit upper bound of $\frac{l}{n}$ as required.

So we can focus on the case when $\operatorname{deg}(\Delta_F-\Lambda_F)< \frac{\epsilon}{2}$. In this case for any $P_F\geq 0$, we have $\Delta_F-\Lambda_F\geq 0$ and the sum of all coefficients of the components are less than $\frac{\epsilon}{2}$, which implies every components in $\Delta_F-\Lambda_F$ are less than $\frac{\epsilon}{2}$. We conclude that $(F,\Delta_F)$ is $\frac{\epsilon}{2}$-lc as each coefficient of the components are less than $(1-\epsilon)+\frac{\epsilon}{2}=1-\frac{\epsilon}{2}$ and $\Delta_F\geq 0$.\\

\emph{Step 6.} In this step we use Proposition \ref{prop: explicit coefficient descent} to bound $\frac{l}{n}$ from above. Let $E_F:=E|_F$, $R_F:=R|_F$, and $N_F:=N|_F$. By Step 5 we see that the conditions of Proposition \ref{prop: explicit coefficient descent} are satisfied, so by taking $\Phi =\mathbb{Z}$ and $\Phi = [\delta,+\infty)$ we can see there exists a natural number $q$ depending only on $\frac{\epsilon}{2}$ such that $qE_F$ is integral and $qR_F\geq \delta$. Moreover, we know that $q$ can be taken as $(\lceil\frac{2}{\epsilon}\rceil!)^2$. Thus $qN_F =qE_F+qR_F$ is the sum of a pseudo-effective integral divisor and an effective divisor with coefficients $\geq \delta$.

By construction, we have $qnN_F-K_F$ is big. We are interested in when will $|pqnN_F|$ defines a birational map. As we are in the situation of smooth curves, this is very easy, since
$$ |pqnN_F| = |K_F+(p-1)qnN_F +qnN_F-K_F|$$ it suffice to let $\operatorname{deg}_F(p-1)qnN_F >2$. Since $qE_F$ is integral and pseudo-effective while $qR_F\geq \delta$, $\operatorname{deg}_F(qE_F)\in \mathbb{N}$ and $\operatorname{deg}_F(qR_F)\geq \frac{1}{\delta}$. It suffice to take $p=3+3\lceil\frac{1}{\delta}\rceil$. Now $\operatorname{Vol}(pqnN_F)\geq 1$, hence $$ \operatorname{Vol}(5pqnN|_G)= \operatorname{Vol}(5pqnN_F)\geq 5>4$$ which by definition we have $l\leq 5pqn$. In particular, we have $\frac{l}{n}\leq 5pq$ is explicitly bounded from above.\\

\emph{Step 7.} In this final step we write down $v$ explicitly. By Step 1, we can always take $v=64u^2$ for $u$ being the upper bound for $\frac{m}{n}$. By Step 2, in the correspongding situation we can take $$u=2+\lceil\frac{1}{\delta}\rceil$$. By Step 3, we can take $$u=3a+1+\lceil\frac{1}{\delta}\rceil$$ where $a$ is a upper bound of $\frac{l}{n}$. By Step 4 and Step 6, we can take $$a = \max\{\frac{64}{\epsilon},15(1+\lceil\frac{1}{\delta}\rceil)(\lceil\frac{2}{\epsilon}\rceil!)^2\}.$$ It's evident that $3a+1>2$, so we can take $$u= \max\{\frac{192}{\epsilon}+1+ \lceil\frac{1}{\delta}\rceil,(1+\lceil\frac{1}{\delta}\rceil)(45(\lceil\frac{2}{\epsilon}\rceil!)^2+1)\}$$ and thus $$v=64(\max\{\frac{192}{\epsilon}+1+ \lceil\frac{1}{\delta}\rceil,(1+\lceil\frac{1}{\delta}\rceil)(45(\lceil\frac{2}{\epsilon}\rceil!)^2+1)\})^2$$.

\end{proof}

\begin{remark}\label{specialness of dim=2}
    It's interesting to ask where does the proof of Lemma  \ref{lem: effective birationality for polarized surfaces} rely on the dimension two situation. The most important part is during the proof of Step 5, we are able to reduce to the case where $(F,\Delta_F)$ is $\frac{\epsilon}{2}$-lc. This requires to use the simple structure of curves, so that bounding singularities is equivalent to bounding coefficients, while the latter can be done by taking degrees. If one checks the original proof of \cite{Bir23} Proposition 4.5, he used a much more difficult approach to bound the singularity, which involves a new constant $t$ derived from an abstract bounded family, making it very difficult to write down explicitly. On the other hand, for the rest of the proof, if one wish to give an explicit bound for higher dimension varieties, one can study the proof of \cite{Bir23} Proposition 4.5 carefully and hope to use some type of inductive treatment to gain an explicit expression.
\end{remark}

\begin{remark}\label{rmk no fraction part}
    When $N$ has no fractional part, i.e. $N$ is a big and nef integral divisor, we can take $\delta$ to be arbitrarily large, so the term $\lceil\frac{1}{\delta}\rceil$ can be simply ignored. In this case, we can take $v= 64(\max\{\frac{192}{\epsilon}+1,((45(\lceil\frac{2}{\epsilon}\rceil!)^2+1))\})^2$.
\end{remark}

\subsection{Proof of Theorem 1.3}
\begin{proof}
    Take $v$ as in Remark \ref{rmk no fraction part}. By Lemma \ref{lem: lower bounds of volumes}, we have $\operatorname{Vol}(X)\geq \frac{1}{42\cdot 84^{128\cdot42^5+168}}$. By taking $N=2K_X$, if $k$ is the smallest natural number that $|kN|$ defines a birational map, then either $k\leq v$ or $\operatorname{Vol}(kN)\leq v$, in the latter case we have $$ v\geq 4k^2\operatorname{Vol}(X)\geq \frac{4k^2}{42\cdot 84^{128\cdot42^5+168}}.$$ It follows that $2k\leq \sqrt{42\cdot 84^{128\cdot42^5+168}\cdot v}$. Thus we have $2k \leq m_0 :=\max \{2v,\lfloor\sqrt{42\cdot 84^{128\cdot42^5+168}\cdot v\rfloor} \}$. In particular let $m_1 =m_0!$, then $|m_1K_X|$ defines a birational map for all such $X$ as $m_1$ is always a multiple of $2k$. Moreover, $$|m'K_X|=|K_X+5m_1K_X+(m'-5m_1-1)K_X|$$ defines a birational map for all $m'\geq 5m_1+1$ as $5m_1K_X+(m'-5m_1-1)K_X$ is potentially birational and we can apply Lemma \ref{lem: potentially birationality} (1).
\end{proof}

\begin{remark}
    It's worth noting that Lemma \ref{lem: effective birationality for polarized surfaces} can also be used to obtain an explicit upper bound of $m$ in the weak Fano case. Indeed, as we know $I_0K_X$ is always Cartier, $$(-I_0K_X\cdot -K_X)=I_0(K_X^2)=I_0\cdot \operatorname{Vol}(-K_X)\in \mathbb{N}$$ thus $\operatorname{Vol}(-K_X)\geq \frac{1}{I_0}$ is bounded from below. In particular one may apply Lemma \ref{lem: effective birationality for polarized surfaces} to $N=-2K_X$ to obtain an explicit $m_0$ as well. However, using this method can not tell when those $|-mK_X|$ will define a birational morphism, not just a birational map.
\end{remark}

\section{Proof of Theorem 1.4}
\begin{proof}
    We follow the idea of the proof of Lemma 4.11 in \cite{Bir23}. Again, we can avoid using abstract bounded families by some explicit computation.

    \emph{Step 1.} We first make some reductions. Taking a $\bQ$-factorialisation we can assume $X$ is $\bQ$-factorial. We can always assume $K_X$ is not pseudo-effective, otherwise we may simply take $l=1$. Let $t>0$ be the smallest positive real number that $K_X+tN$ is pseudo-effective. It suffice to find an explicit upper bound of $t$, as $(K_X+\lceil1+t\rceil N) $is big over $Z$.\\

    \emph{Step 2.} In this step we run MMP and study the resulting variety. Regard $(X,tN)$ as a generalised pair over $Z$ with nef part $tN$. This pair is generalised $\epsilon$-lc. Since $tN$ is big over $Z$, we can run a MMP over $Z$ on $K_X+tN$ such that it ends with a good minimal model $X'$, on which $K_{X'}+tN'$ is semi-ample over $Z$ \cite{BZ16} Lemma 4.4; here $K_{X'}+tN'$ is just the pushdown of $K_X+tN$. Now $K_{X'}+tN'$ defines a contraction $X'\rightarrow V'/Z$ which is not birational, otherwise $K_{X'}+tN'$ is big which in turn implies $K_X+tN$ is also big, contradicting to the minimality of $t$. In particular, we know $V'$ is either a point or a curve.

    Now as $tN'$ is big over $V'$ and $K_{X'}\equiv-tN'/V'$ is not pseudo-effective over $V'$, hence we can run an MMP$/V'$ on $K_{X'}$ such that it ends with a Mori fibre space $X''\rightarrow W''/V'$. Note that $K_{X'}+tN'$ is MMP-trivial, we have $(X'',tN'')$ is also genearlised $\epsilon$-lc. Thus $(X'',0)$ is also $\epsilon$-lc as $K_{X''}$ is $\bQ$-Cartier.\\

    \emph{Step 3.} In this step we restrict to the general fibres of $X''\rightarrow W''$ and reduce to the $\epsilon$-lc Fano case. Let $F''$ be a general fibre of  $X''\rightarrow W''$. Then $F''$ is an $\epsilon$-lc Fano Variety of dimension at most two. Let $N_{F''}=N|_{F''}$ $E_{F''}=E|_{F''}$, $R_{F''}=R|_{F''}$. We have $N_{F''}=E_{F''}+R_{F''}$ and $$K_{F''}+tN_{F''} = (K_{X''}+tN'')|_{F''} \equiv 0.$$ Moreover $E_{F''}$ is integral and pseudo-effective, while $R_{F''}\geq \delta $. Therefore, after replacing $X,N,E,R$ by $F'',N_{F''} ,E_{F''},R_{F''}$ and replacing $Z$ by a point, we may assume $X$ is an $\epsilon$-lc Fano variety.\\

    \emph{Step 4.} In this step we use explicit BAB of $\epsilon$-lc Fano surfaces to obtain an explicit bound of $t$. Note that now $X$ is either $\bP^1$ or a $\epsilon$-lc Fano surface. In the first case it suffice to take $t=2+2\lceil\frac{1}{\delta}\rceil$. In the second case, we know $-I_0K_X$ is a Cartier ample divisor, where $I_0= \lfloor2\left(\frac{2}{\epsilon}\right)^{\frac{128}{\epsilon^5}}\rfloor !$ as Lemma \ref{lem: index epsilonlc fano surface}. Furthermore, we know such $X$ form a bounded family, and in fact $$\operatorname{Vol}(-K_X)=(K_X^2)\leq \max\{64,\frac{8}{\epsilon}+4\}=:v_0.$$ In particular, since $(E\cdot(-I_0K_X))\geq 1$ if $E$ is not numerically trivial and $(R\cdot(-I_0K_X))\geq \delta$, let $\lambda = \min\{1,\delta\}$, we obtain:

    $$I_0\cdot v_0\geq  I_0(K_X^2)=(tN\cdot (-I_0K_X)) \geq \lambda\cdot t$$.

    Thus $t\leq \frac{I_0\cdot v_0}{\lambda}$ and this gives an explicit $$l=\lceil1+\frac{(\lfloor2\left(\frac{2}{\epsilon}\right)^{\frac{128}{\epsilon^5}}\rfloor !)\cdot \max\{64,\frac{8}{\epsilon}+4\}}{\min\{1,\delta\}}\rceil.$$

\end{proof}

\section{Some Corollaries }

A simple application of Lemma \ref{lem: effective birationality for polarized surfaces} also yields the explicit effective birationality of Calabi-Yau $\epsilon$-lc surfaces:

\begin{corollary}
  Let $\epsilon$ and $v_0$ be two positive real numbers. Then there exists $m_0 := \max \{v,\lfloor\sqrt{v_0\cdot v} \rfloor\}$ where $v=64(\max\{\lceil\frac{192}{\epsilon}\rceil+1+ \lceil\frac{1}{\delta}\rceil,(1+\lceil\frac{1}{\delta}\rceil)(45(\lceil\frac{2}{\epsilon}\rceil!)^2+1)\})^2$ satisfying the following. Assume:
  \begin{enumerate}
      \item $X$ is a projective $\epsilon$-lc surface with $K_X\equiv 0$,
      \item $N$ is a nef and big integral Weil divisor,
      \item $\operatorname{Vol}(N)\geq v_0$.
  \end{enumerate}
  Then there exists some $m^*\leq m_0$ such that $|m^*N|$ defines a birational map. In particular, $|m_0!N|$ defines a birational map for all such $X$ and $N$. Moreover, for every $m'\geq 5(m_0!)$, $|m'N|$ defines a birational map.
\end{corollary}
\begin{proof}
    This is a direct application of Lemma \ref{lem: effective birationality for polarized surfaces}.
\end{proof}

\begin{corollary}
    Let $\epsilon$ be a positive real number. Then there exists $v=64(\max\{\lceil\frac{192}{\epsilon}\rceil+1+ \lceil\frac{1}{\delta}\rceil,(1+\lceil\frac{1}{\delta}\rceil)(45(\lceil\frac{2}{\epsilon}\rceil!)^2+1)\})^2$ satisfying the following. Assume:
\begin{enumerate}
    \item $X$ is a projective $\epsilon$-lc surface,
    \item $N$ is a nef and big Cartier divisor on $X$,
    \item $N-K_X$ and $N+K_X$ are big.
\end{enumerate}
 If $m_0$ is the smallest natural number such that $|m_0N|$ defines a birational map, then $m_0\leq v$. In particular, $|v!N|$ defines a birational map for all such $N$ and $X$. Moreover, for every $m'\geq 5(v!)+1$, $|m'N|$ defines a birational map.
\end{corollary}

\begin{proof}
    Since $N$ is a big and nef Cartier divisor, we always have $\operatorname{Vol}(N)=(N^2)\geq 1$. Thus the assertion follows by Lemma \ref{lem: effective birationality for polarized surfaces}. Note that we have only assumed $N-K_X$ to be big, if moreover we know that $N-K_X$ is also nef, then by effective base-point-free theorem (\cite[Theorem 1.1, Remark 1.2]{Fuj09} \cite[1.1 Theorem]{Kol93}), we can see $|96N|$ is base-point-free, thus $|96v!N|$ defines a birational morphism.
\end{proof}

Another important observation is that using Theorem \ref{pseudo effective threshold} one can  remove the condition $K_X+N$ is big. Indeed, after replacing $N$ by a bounded multiple $lN$, we are reduced to the case where $K_X+lN$ is big.

\begin{corollary}
    Let $\epsilon$, $\delta$ be two positive real numbers. Then there exists $$l=\lceil1+\frac{(\lfloor2\left(\frac{2}{\epsilon}\right)^{\frac{128}{\epsilon^5}}\rfloor !)\cdot \max\{64,\frac{8}{\epsilon}+4\}}{\min\{1,\delta\}}\rceil$$ and $$v=64(\max\{\lceil\frac{192}{\epsilon}\rceil+1+ \lceil\frac{1}{l\delta}\rceil,(1+\lceil\frac{1}{l\delta}\rceil)(45(\lceil\frac{2}{\epsilon}\rceil!)^2+1)\})^2$$ satisfying the following. Assume:
\begin{enumerate}
    \item $X$ is a projective $\epsilon$-lc surface,
    \item $N$ is a nef and big $\bQ$-divisor on $X$,
    \item $N-K_X$ is big, and
    \item $N=E+R$ where $E$ is integral and pseudo-effective and $R\geq 0$ with coefficients in ${0}\cup [\delta,+\infty)$.
\end{enumerate}
 If $m$ is the smallest natural number such that $|mlN|$ defines a birational map, then either $m\leq v$ or $\operatorname{Vol}(mlN)\leq v$.

\end{corollary}

\begin{proof}
    It suffice to replace $N$ by $lN$ and replace $\delta$ by $l\delta$ where $l$ is given by Theorem \ref{pseudo effective threshold}. 
\end{proof}

\section{Approach to the threefold case}

In this section we briefly explain how one may try to get an explicit bound for polarized threefolds. We also explain the technical difficulties in this situation.\\

First suppose we have a weak Fano $\epsilon$-lc $3$-fold $X$. In this case if we wish to use similar argument as the proof of Theorem \ref{thm: weak fano surface explicit birationality}, we first should find an explicit natural number $I_3$ such that $I_3K_X$ is Cartier for all $X$. This seems not very practical for $3$-folds. Moreover, even if we find such $I_3$ and then use effective base-point-free to find some $|-n_3I_3K_X|$ is base-point-free, it's not easy to tell when will $|-kn_3I_3K_X|$ be birational. This is because in the threefold case we don't have a precise analogue of Theorem \ref{thm: effective birationality with h^0 > 2}. In fact, the results of Minzhe Zhu in \cite{Zhu23} is only valid for $n$-folds $X$ with those $H$ such that $\operatorname{dim}\overline{\phi_H(X)}\geq n-1 $. In particular, for $3$-folds, we still need to consider the case where $\operatorname{dim}\overline{\phi_H(X)}= n-2=1 $.\\

Now suppose we are in a more general setting of a polarized threefold $X$ with a integral big and nef Weil divisor $N$ such that $N-K_X$ is big. In order to find some explicit $m$ such that $|mN|$ defines a birational map, one may wish to prove an analogue of Lemma \ref{lem: potentially birationality} for $3$-folds. However, as explained in \ref{specialness of dim=2}, such an approach faces some difficulties as we may inevitably using some abstract bounded family during the proof of Step 5, without introducing some alternative methods. On the other hand, little is known for the effective lower bound of the volume of $X$ when $X$ is a log canonical $3$-fold.

\section{Further remarks}

\begin{remark}[Reasonable and optimal bounds]
It's obvious our explicit bound is far from being reasonable, not to mention being optimal. On the one hand, if one can get a better bound for $I=I(\epsilon)$ in Lemma \ref{lem: index epsilonlc fano surface}, then the bound of $m$ in Theorem \ref{thm: weak fano surface explicit birationality} may be greatly improved. On the other hand, if we have a reasonable lower bound of the volume of log canonical surfaces, it will greatly improve the bound of $m$ in \ref{thm: general type surface explicit birationality}. iN fact, it's expected that the lower bound is close, or equal to $\frac{1}{48963}$. Moreover, every time we use factorial in the expression of the outcome, it's only used for simplicity. In practice if we fix a particular $\epsilon$, we can take a much smaller number $m$ without using factorials.
\end{remark}

\begin{remark}[Explicit bound for pairs]\label{rem: explicit bound for pairs}
One may also ask whether we can find an explicit bounded of $m$ for polarized surface pairs $(X,B)$. 

Suppose we are given a DCC set $\Phi\subset[0,1]$ of rational numbers. Assume 
\begin{enumerate}
    \item $(X,B)$ is a klt projective pair of dimension two,
    \item the coefficients of $B$ are in $\Phi$,
    \item $N$ is a nef and big integral divisor, and
    \item $N-K_X-B$ and $K_X+B$ are pseudo-effective.
\end{enumerate}

The question is to find some $m$ explicitly depending only on $\Phi$ such that $|m'N|$ or $|K_X+m'N|$ defines a birational map for all $m'\geq m$. Interesting cases are when $\Phi$ is finite or $\Phi =\{1-\frac{1}{n}|n\in \mathbb{N\}}$ is the hyper-standard set. It's expected that the finite case is easier, while the hyper-standard case is closely related to the problem of explicit boundedness of birational automorphism groups of surfaces of general type, c.f. \cite{HMX13} for related discussions.

\end{remark}

\end{document}